\documentclass[reqno]{amsart}
\usepackage{hyperref}

\usepackage{amsmath,amssymb,amsthm}
\usepackage{enumerate}
\usepackage{color}
\usepackage{epsfig}
\usepackage{graphics}
\usepackage{graphicx}
\usepackage{subfigure}

\begin{document}
\title[]{Essential Spectra of Quasi-parabolic Composition Operators on Hardy Spaces of the Poly-disc}% The title of the article
\author[u. g\"{u}l]{u\u{g}ur g\"{u}l}

\address{u\u{g}ur g\"{u}l,  \newline
Hacettepe University, Department of Mathematics, 06800, Beytepe,
Ankara, TURKEY}
\email{\href{mailto:gulugur@gmail.com}{gulugur@gmail.com}}
%\href{mailto:myaddress@wikibooks.org}{myaddress@wikibooks.org}
%\email{{mailto:aozbekler@gmail.com}{aozbekler@gmail.com}}

\thanks{Submitted XXX 2012}

\begin{abstract}
This work is a generalization of the results in [Gul] to bi-disc
case. As in [Gul], quasi-parabolic composition operators on the
Hilbert-Hardy space of the bi-disc are written as a linear
combination of Toeplitz operators and Fourier multipliers. The
C*-algebra generated by Toeplitz operators and Fourier multipliers
on the Hilbert-Hardy space of the bi-disc is written as the tensor
product of the similar C*-algebra in one variable with itself. As a
result we find a nontrivial set lying inside the essential spectra
of quasi-parabolic composition operators.
\end{abstract}

\maketitle
%\numberwithin{equation}{section}
\newtheorem{theorem}{Theorem}
\newtheorem{acknowledgement}[theorem]{Acknowledgement}
\newtheorem{algorithm}[theorem]{Algorithm}
\newtheorem{axiom}[theorem]{Axiom}
\newtheorem{case}[theorem]{Case}
\newtheorem{claim}[theorem]{Claim}
\newtheorem{conclusion}[theorem]{Conclusion}
\newtheorem{condition}[theorem]{Condition}
\newtheorem{conjecture}[theorem]{Conjecture}
\newtheorem{corollary}[theorem]{Corollary}
\newtheorem{criterion}[theorem]{Criterion}
\newtheorem{definition}[theorem]{Definition}
\newtheorem{example}[theorem]{Example}
\newtheorem{exercise}[theorem]{Exercise}
\newtheorem{lemma}[theorem]{Lemma}
\newtheorem{notation}[theorem]{Notation}
\newtheorem{problem}[theorem]{Problem}
\newtheorem{proposition}[theorem]{Proposition}
\newtheorem{remark}[theorem]{Remark}
\newtheorem{solution}[theorem]{Solution}
\newtheorem{summary}[theorem]{Summary}
\newtheorem*{thma}{Main Theorem 1}
\newtheorem*{thmb}{Main Theorem 2}
\newtheorem*{thmc}{Theorem C}
\newtheorem*{thmd}{Theorem D}
\newcommand{\norm}[1]{\left\Vert#1\right\Vert}
\newcommand{\abs}[1]{\left\vert#1\right\vert}
\newcommand{\set}[1]{\left\{#1\right\}}
\newcommand{\Real}{\mathbb R}
\newcommand{\eps}{\varepsilon}
\newcommand{\To}{\longrightarrow}
\newcommand{\BX}{\mathbf{B}(X)}
\newcommand{\A}{\mathcal{A}}

\section*{introduction}

Quasi-parabolic composition operators is a generalization of the
composition operators induced by parabolic linear fractional
non-automorphisms of the unit disc that fix a point $\xi$ on the
boundary. These linear fractional transformations for $\xi= 1$ take
  the form
  \[\varphi_{a}(z)=\frac{2iz+a(1-z)}{2i+a(1-z)}\]
  with $\Im(a) > 0$. Quasi-parabolic composition operators on $H^{2}(\mathbb{D})$
  are composition operators induced by the symbols where `$a$' is replaced by a bounded
  analytic function `$\psi$' for which $\Im(\psi(z))>\delta>0$\quad
  $\forall z\in\mathbb{D}$. We recall that the set of cluster points $\mathcal{C}_{\xi}(\psi)$ of $\psi\in
  H^{\infty}(\mathbb{D})$ at $\xi\in\mathbb{T}$ is defined to be the set of points
  $z\in$ $\mathbb{C}$ for which there is a sequence $\{z_{n}\}\subset$
  $\mathbb{D}$ so that $z_{n}\rightarrow$ $\xi$ and
  $\psi(z_{n})\rightarrow$ $z$. Similarly for the bi-disc, $\mathcal{C}_{(\xi_{1},\xi_{2})}(\psi)$ of $\psi\in
  H^{\infty}(\mathbb{D}^{2})$ at $(\xi_{1},\xi_{2})\in\mathbb{T}^{2}$ is defined to be the set of points
  $z\in$ $\mathbb{C}$ for which there is a sequence $\{z_{n}\}\subset$
  $\mathbb{D}^{2}$ so that $z_{n}\rightarrow$ $(\xi_{1},\xi_{2})$ and
  $\psi(z_{n})\rightarrow$ $z$. In [Gul] we showed that if $\psi\in
  QC(\mathbb{T})$ then these composition operators are essentially
  normal and their essential spectra are given as
  \[\sigma_{e}(C_{\varphi})=\{e^{izt}:t\in
  [0,\infty],z\in\mathcal{C}_{1}(\psi)\}\cup\{0\}\]
where\quad $\mathcal{C}_{1}(\psi)$ is the set of cluster points of
$\psi$ at $1$.

In this work we investigate this phenomenon in the bi-disc setting.
We look at the composition operators on the Hardy space of the
bi-disc induced by symbols of the form
\[\varphi(z_{1},z_{2})=\left(\frac{2iz_{1}+\psi_{1}(z_{1},z_{2})(1-z_{1})}{2i+\psi_{1}(z_{1},z_{2})(1-z_{1})},\frac{2iz_{2}+\psi_{2}(z_{1},z_{2})(1-z_{2})}{2i+\psi_{2}(z_{1},z_{2})(1-z_{2})}\right)\]
where $\psi_{1},\psi_{2}\in H^{\infty}(\mathbb{D}^{2})$ such that
$\Im(\psi_{j}(z_{1},z_{2}))>\delta>0$\quad $\forall
(z_{1},z_{2})\in\mathbb{D}^{2}$, $j=1,2$. These symbols are carried
over via the Cayley transform to the symbols of the form
\[\tilde{\varphi}(w_{1},w_{2})=(w_{1}+\psi_{1}\circ\mathfrak{C}_{2}(w_{1},w_{2}),w_{2}+\psi_{2}\circ\mathfrak{C}_{2}(w_{1},w_{2}))\]
on the two dimensional upper half-plane $\mathbb{H}^{2}$,i.e.
\[\mathfrak{C}_{2}^{-1}\circ\varphi\circ\mathfrak{C}_{2}=\tilde{\varphi}\]
where
\[\mathfrak{C}_{2}(z_{1},z_{2})=\left(\frac{z_{1}-i}{z_{1}+i},\frac{z_{2}-i}{z_{2}+i}\right)\]
is the Cayley transform. In particular we prove the following
result:
\begin{thmb}
Let $\varphi:\mathbb{D}^{2}\rightarrow$ $\mathbb{D}^{2}$ be an
analytic self-map of $\mathbb{D}^{2}$ such that
$$\varphi(z_{1},z_{2})=\left(\frac{2iz_{1}+\psi_{1}(z_{1},z_{2})(1-z_{1})}{2i+\psi_{1}(z_{1},z_{2})(1-z_{1})},\frac{2iz_{2}+\psi_{2}(z_{1},z_{2})(1-z_{2})}{2i+\psi_{2}(z_{1},z_{2})(1-z_{2})}\right) $$
where $\psi_{j}\in$ $H^{\infty}(\mathbb{D}^{2})$ with
$\Im(\psi_{j}(z_{1},z_{2}))
> \epsilon> 0$ for all $(z_{1},z_{2})\in$ $\mathbb{D}^{2}$, $j=1,2$. Then\\
$C_{\varphi}:$ $H^{2}(\mathbb{D}^{2})\rightarrow$
$H^{2}(\mathbb{D}^{2})$ is bounded. Moreover if $\psi_{j}\in$
$(QC\otimes QC)\cap H^{\infty}(\mathbb{D}^{2})$ then we have
\[\sigma_{e}(C_{\varphi})\supseteq\{e^{i(z_{1}t_{1}+z_{2}t_{2})}:t_{1},t_{2}\in [0,\infty],z_{1}\in\mathcal{C}_{(1,1)}(\psi_{1})\textrm{and}\quad z_{2}\in\mathcal{C}_{(1,1)}(\psi_{2})\}\cup\{0\},\]
where\quad $\mathcal{C}_{(1,1)}(\psi)$ is the set of cluster points
of $\psi$ at $(1,1)\in\mathbb{T}^{2}$.
\end{thmb}
We work on the two dimensional upper half-plane $\mathbb{H}^{2}$
and use Banach algebra techniques to compute the essential spectra
of operators that correspond to ``quasi-parabolic" operators. As
in [Gul] translation operators on $H^{2}(\mathbb{H}^{2})$ can be
considered as Fourier multipliers on $H^{2}(\mathbb{H}^{2})$ where
$\mathbb{H}^{2}$ is considered as a tubular domain (we refer the
reader to [Up] for the definition and properties of fourier
multipliers on Hardy and Bergman spaces of tubular domains in
several complex variables). Throughout the present work,
$H^{2}(\mathbb{H}^{2})$ will be considered as a closed subspace of
$L^{2}(\mathbb{R}^{2})$ via the boundary values. With the help of
Cauchy integral formula we prove an integral formula that gives
composition operators as integral operators. Using this integral
formula we show that operators that correspond
to``quasi-parabolic" operators fall in a C*-algebra generated by
Toeplitz operators and Fourier multipliers.

The remainder of this paper is organized as follows: In section 1
we give the basic definitions and preliminary material that we
will use throughout. For the benefit of the reader we explicitly
recall some facts about C*-algebras, tensor products of
C*-algebras and nuclear C*-algebras. Using a version of
Paley-Wiener theorem due to Bochner we also introduce the
C*-algebra of Fourier multipliers acting on
$H^{2}(\mathbb{H}^{2})$. In Section 2 we first show that
``quasi-parabolic" composition operators are bounded on
$H^{2}(\mathbb{H}^{2})$ and prove an integral representation
formula for composition operators on $H^{2}(\mathbb{H}^{2})$. Then
we use this integral formula together with the boundedness result
to prove that a ``quasi-parabolic" composition operator is written
as a series of Toeplitz operators and Fourier multipliers which
converges in operator norm. In section 3 we analyze the C*-algebra
generated by Toeplitz operators with $QC(\mathbb{R})\otimes
QC(\mathbb{R})$ symbols and Fourier multipliers modulo compact
operators. We write this C*-algebra as the tensor product of the
C*-algebra $\Psi$ in [Gul] with itself. In doing this we follow
the approach taken by [DoH] for analyzing the Toeplitz C*-algebra
of the bi-disc. We use this tensor product to identify the
character space of the C*-algebra generated by Toeplitz operators
with $QC(\mathbb{R})\otimes QC(\mathbb{R})$ symbols and Fourier
multipliers modulo compact operators. In section 4, using the
machinery developed in sections 2 and 3, we obtain some results
about the essential spectra of ``quasi-parabolic" composition
operators.

\section{preliminaries}

In this section we fix the notation that we will use throughout
and recall some preliminary facts that will be used in the sequel.

Let $S$ be a compact Hausdorff topological space. The space of all
complex valued continuous functions on $S$ will be denoted by
$C(S)$. For any $f\in C(S)$, $\parallel f\parallel_{\infty}$ will
denote the sup-norm of $f$, i.e. $$\parallel
f\parallel_{\infty}=\sup\{\mid f(s)\mid:s\in S\}.$$ For a Banach
space $X$, $K(X)$ will denote the space of all compact operators
on $X$ and $B(X)$  will denote the space of all bounded linear
operators on $X$. The open unit disc will be denoted by
$\mathbb{D}$, the open upper half-plane will be denoted by
$\mathbb{H}$, the real line will be denoted by $\mathbb{R}$ and
the complex plane will be denoted by $\mathbb{C}$. The one point
compactification of $\mathbb{R}$ will be denoted by
$\dot{\mathbb{R}}$ which is homeomorphic to $\mathbb{T}$. For any
$z\in$ $\mathbb{C}$, $\Re(z)$ will denote the real part, and
$\Im(z)$ will denote the imaginary part of $z$, respectively. For
any subset $S\subset$ $B(H)$, where $H$ is a Hilbert space, the
C*-algebra generated by $S$ will be denoted by $C^{*}(S)$ and for
any subset $S\subset A$ where $A$ is a C*-algebra, the closed
two-sided ideal generated by $S$ will be denoted by $I^{*}(S)$.

The Hardy space of the bi-disc $H^{2}(\mathbb{D}^{2})$ is
identified as the tensor product of the two copies of the
classical Hardy space of the unit disc $H^{2}(\mathbb{D})$, i.e.
the closure of the linear span of the set of functions
$$\{h(z,w)=f(z)g(w):f,g\in H^{2}(\mathbb{D})\}$$ with respect to
the inner product
$$\langle h_{1},h_{2}\rangle=\int_{0}^{2\pi}\int_{0}^{2\pi}h_{1}(e^{i\theta_{1}},e^{i\theta_{2}})\overline{h_{2}(e^{i\theta_{1}},e^{i\theta_{2}})}d\theta_{1}d\theta_{2}.$$
In the same way the Hardy space of the two dimensional half-plane
$H^{2}(\mathbb{H}^{2})$ is identified as the tensor product of the
two copies of the Hardy space of the upper half-plane
$H^{2}(\mathbb{H})$. Note that $H^{2}(\mathbb{D}^{2})$ and
$H^{2}(\mathbb{H}^{2})$ are isometrically isomorphic. An isometric
isomorphism $\Phi:H^{2}(\mathbb{D}^{2})\rightarrow$
$H^{2}(\mathbb{H}^{2})$ is given by
\begin{equation*}
(\Phi
f)(z_{1},z_{2})=\bigg(\frac{1}{z_{1}+i}\bigg)\bigg(\frac{1}{z_{2}+i}\bigg)f\bigg(\frac{z_{1}-i}{z_{1}+i},\frac{z_{2}-i}{z_{2}+i}\bigg)
\end{equation*}
Under this isometric isomorphism $C_{\varphi}$ for an analytic
self-map $\varphi:\mathbb{D}^{2}\rightarrow\mathbb{D}^{2}$ is
carried over to
$(\frac{(\tilde{\varphi}_{1}(z_{1},z_{2})+i)(\tilde{\varphi}_{2}(z_{1},z_{2})+i)}{(z_{1}+i)(z_{2}+i)})C_{\tilde{\varphi}}$
on $H^{2}(\mathbb{H}^{2})$ through $\Phi$, where
$\tilde{\varphi}=$
$\mathfrak{C}_{2}^{-1}\circ\varphi\circ\mathfrak{C}_{2}$, i.e.we
have
\begin{equation}
\Phi C_{\varphi}\Phi^{-1} =
    T_{(\frac{(\tilde{\varphi}_{1}(z_{1},z_{2})+i)(\tilde{\varphi}_{2}(z_{1},z_{2})+i)}{(z_{1}+i)(z_{2}+i)})}C_{\tilde{\varphi}}
\end{equation}
A tubular domain $\Pi=X\oplus i\Lambda$ is a domain in
$\mathbb{C}^{n}$ where $\Lambda\subseteq\mathbb{R}^{n}$ is a cone
i.e. $x$,$y\in\Lambda$ $\Rightarrow$ $x+y\in\Lambda$ and $\forall
t>0$, $x\in\Lambda$, $tx\in\Lambda$. We observe that
$\mathbb{H}^{2}=\mathbb{R}^{2}\oplus i(\mathbb{R}^{+})^{2}$ is a
tubular domain. We have the following Paley-Wiener type theorem
due to Bochner (see [Up], p.93):
\begin{theorem}
Let $\Pi=X\oplus i\Lambda$ be a tubular domain where
$\Lambda\subseteq\mathbb{R}^{n}$ is a cone and
$X\cong\mathbb{R}^{n}$ then the Fourier transform
\begin{equation}
\mathcal{F}(f)(x)=\frac{1}{(2\pi)^{\frac{n}{2}}}\int_{\mathbb{R}^{n}}f(t)e^{-ix.t}dt
\end{equation}
maps $H^{2}(\Pi)$ isometrically onto $L^{2}(\Lambda^{*})$ where
$\Lambda^{*}=\{y\in\mathbb{R}^{n}:x.y>0\quad\forall x\in\Lambda\}$
is the dual cone of $\Lambda$.
\end{theorem}
Since $\mathbb{H}^{2}=\mathbb{R}^{2}\oplus i(\mathbb{R}^{+})^{2}$
and $((\mathbb{R}^{+})^{2})^{*}=(\mathbb{R}^{+})^{2}$, Bochner's
theorem gives us that
$$\mathcal{F}:H^{2}(\mathbb{H}^{2})\rightarrow
L^{2}((\mathbb{R}^{+})^{2})$$ is an isometric isomorphism.

Using Bochner's theorem we define the following class of operators
on $H^{2}(\mathbb{H}^{2})$ which we call ``Fourier Multipliers": let
$\vartheta\in C_{0}((\mathbb{R}^{+})^{2})$ then $D_{\vartheta}$
defined in the following way
\begin{equation*}
D_{\vartheta}=\mathcal{F}^{-1}M_{\vartheta}\mathcal{F}
\end{equation*}
maps $H^{2}(\mathbb{H}^{2})$ into itself. Let
\begin{equation*}
F_{C_{0}((\mathbb{R}^{+})^{2})}=\{D_{\vartheta}:\vartheta\in
C_{0}((\mathbb{R}^{+})^{2})\}
\end{equation*} then $F_{C_{0}((\mathbb{R}^{+})^{2})}$ is a
commutative C*-algebra of operators on $H^{2}(\mathbb{H}^{2})$ and
$$F_{C_{0}((\mathbb{R}^{+})^{2})}\cong C_{0}((\mathbb{R}^{+})^{2}).$$

\bigskip

For any Banach algebra $A$ let $M(A)$ denote the space of
characters of $A$ i.e. $$M(A)=\{x\in A^{*}:x(ab)=x(a)x(b)\}.$$
where $A^{*}$ is the dual space of $A$. If $A$ has identity then
$M(A)$ is a compact Hausdorff topological space with the weak*
topology. If $A$ is commutative then $M(A)$ coincides with the
maximal ideal space of $A$. If $A$ is a C*-algebra and $I$ is a
two-sided closed ideal of $A$, then the quotient algebra $A/I$ is
also a C*-algebra.For a Banach algebra $A$, we denote by $com(A)$
the two-sided closed ideal in $A$ generated by the commutators
$\{a_{1}a_{2}-a_{2}a_{1}:a_{1},a_{2}\in A\}$. It is not difficult
to see that
\begin{equation}
M(A/I)=M(A)
\end{equation}
for any closed two-sided ideal $I\subseteq$ $com(A)$ since any
character $\phi$ is zero on $com(A)$.For $a\in A$ the spectrum
$\sigma_{A}(a)$ of $a$ on $A$ is defined as
\begin{equation*}
    \sigma_{A}(a)=\{\lambda\in\mathbb{C}:\lambda e-a\ \ \textrm{is not invertible in}\ A\},
\end{equation*}
where $e$ is the identity of $A$. We will use the spectral
permanency property of C*-algebras (see [Rud], pp. 283); i.e. if
$A$ is a C*-algebra with identity and $B$ is a closed *-subalgebra
of $A$, then for any $b\in B$ we have
\begin{equation}
\sigma_{B}(b)=\sigma_{A}(b).
\end{equation}
To compute essential spectra we employ the following important
fact (see [Rud], pp. 268): If $A$ is a commutative Banach algebra
with identity then for any $a\in A$ we have
\begin{equation*}
    \sigma_{A}(a)=\{x(a):x\in M(A)\}.
\end{equation*}
In general (for $A$ not necessarily commutative), we have
\begin{equation}
    \sigma_{A}(a)\supseteq\{x(a):x\in M(A)\}.
\end{equation}
Let $H$ and $K$ be two given Hilbert spaces. On the algebraic
tensor product $H\otimes K$ of $H$ and $K$, there is a unique
inner product $\langle.,.\rangle$ satisfying the following
equation
\begin{equation*}
\langle x_{1}\otimes y_{1},x_{2}\otimes y_{2}\rangle=\langle
x_{1},x_{2}\rangle_{H}\langle y_{1},y_{2}\rangle_{K}
\end{equation*}
$\forall x_{1},x_{2}\in H\quad y_{1},y_{2}\in K$ (See [Mur]
pp.185). For any $T\in B(H)$ and $S\in B(K)$ there is a unique
operator $T\hat{\otimes} S\in$ $B(H\otimes K)$ satisfying the
following equation:
\begin{equation*}
(T\hat{\otimes} S)(x\otimes y)=Tx\otimes Sy
\end{equation*}
Moreover $\parallel T\hat{\otimes} S\parallel=$ $\parallel
T\parallel\parallel S\parallel$ (See [Mur] pp.187). For any two
C*-algebras $A\subset B(H)$ and $B\subset B(K)$ the algebraic
tensor product $A\odot B$ is defined to be the linear span of
operators of the form $T\hat{\otimes} S$ i.e.
\begin{equation*}
A\odot B=\{\sum_{j=1}^{n}T_{j}\hat{\otimes} S_{j}:T_{j}\in A,\quad
S_{j}\in B\}
\end{equation*}
The algebraic tensor product $A\odot B$ becomes a *-algebra with
multiplication
\begin{equation*}
(x_{1}\otimes y_{1})(x_{2}\otimes y_{2})= x_{1}x_{2}\otimes
y_{1}y_{2}
\end{equation*}
and involution
\begin{equation*}
(x\otimes y)^{*}=x^{*}\otimes y^{*}
\end{equation*}
However there might be more than one norm making the closure of
$A\odot B$ into a C*-algebra. If $\gamma$ is a pre C*-algebra norm
on $A\odot B$ we denote by $A\otimes_{\gamma} B$ the closure of
$A\odot B$ with respect to this pre C*-algebra norm $\gamma$.
 A C*-algebra $A$ is called ``nuclear" if for any C*-algebra
$B$ there is a unique pre C*-algebra norm on the algebraic tensor
product $A\odot B$ of $A$ and $B$. A well-known theorem of
Takesaki asserts that any commutative C*-algebra is nuclear (see
[Mur] p.205). An extension of a C*-algebra by nuclear C*-algebras
is nuclear, i.e. if $A$, $B$ and $C$ are C*-algebras s.t. the
following sequence
\begin{equation*}
0\xrightarrow{} A\xrightarrow{j} B\xrightarrow{\pi} C\xrightarrow{}
0
\end{equation*}
is short exact and $A$ and $C$ are nuclear then $B$ is also
nuclear (see [Mur] p.212). For any separable Hilbert space $H$ the
C*-algebra of all compact operators $K(H)$ on $H$ is nuclear (see
[Mur] pp. 183 and 196). For any separable Hilbert spaces $H_{1}$
and $H_{2}$ we have
\begin{equation}
K(H_{1}\otimes H_{2})=K(H_{1})\otimes K(H_{2})
\end{equation}
(See [DoH] pp.207).
 We recall the
following fact about tensor products of C*-algebras: If $A$ and
$B$ are C*-algebras then we have
\begin{equation}
M(A\otimes B)\cong M(A)\times M(B)
\end{equation}
that is the map
$(\phi_{1},\phi_{2})\rightarrow\phi_{1}\hat{\otimes}\phi_{2}$
where
\begin{equation}
(\phi_{1}\hat{\otimes}\phi_{2})(a\otimes b)=\phi_{1}(a)\phi_{2}(b)
\end{equation}
is a homeomorphism of $M(A)\times M(B)$ onto $M(A\otimes B)$. See
[Mur] pp.189. The essential spectrum $\sigma_{e}(T)$ of an
operator $T$ acting on a Banach
  space $X$ is the spectrum of the coset of $T$ in the Calkin algebra
  $\mathcal{B}(X)/K(X)$, the algebra of bounded linear operators modulo
  compact operators. The well known Atkinson's theorem identifies the essential
  spectrum of $T$ as the set of all $\lambda\in$ $\mathbb{C}$ for
  which $\lambda I-T$ is not a Fredholm operator. The essential norm of $T$ will be denoted by $\parallel T\parallel_{e}$ which is defined as
\begin{equation*}
 \parallel T\parallel_{e}=\inf\{\parallel T+K\parallel:K\in K(X)\}
\end{equation*}
   The bracket $[\cdot]$ will denote the equivalence class modulo
  $K(X)$.
  Using the isometric isomorphism $\Phi$, one may transfer Fatou's
theorem in the bi-disc case to two dimensional upper half-plane
and may embed
   $H^{2}(\mathbb{H}^{2})$ in $L^{2}(\mathbb{R}^{2})$ via
   $f\longrightarrow$ $f^{*}$ where $f^{*}(x_{1},x_{2})=$
   $\lim_{y\rightarrow 0}f(x_{1}+iy,x_{2}+iy)$. This embedding is an
   isometry.

Throughout the paper, using $\Phi$, we will go back and forth
between $H^{2}(\mathbb{D}^{2})$ and $H^{2}(\mathbb{H}^{2})$. We
use the property that $\Phi$ preserves spectra, compactness and
essential spectra i.e. if $T\in B(H^{2}(\mathbb{D}^{2}))$ then
\begin{equation*}
    \sigma_{B(H^{2}(\mathbb{D}^{2}))}(T)=\sigma_{B(H^{2}(\mathbb{H}^{2}))}(\Phi\circ
    T\circ\Phi^{-1}),
\end{equation*}
$K\in K(H^{2}(\mathbb{D}^{2}))$ if and only if $\Phi\circ
K\circ\Phi^{-1} \in K(H^{2}(\mathbb{H}^{2}))$ and hence we have
\begin{equation}
    \sigma_{e}(T)=\sigma_{e}(\Phi\circ T\circ\Phi^{-1}).
\end{equation}
The local essential range $\mathcal{R}_{\infty}(\psi)$ of $\psi\in
L^{\infty}(\mathbb{R})$ at $\infty$
  is defined as the set of points $z\in$ $\mathbb{C}$ so that, for all
$\varepsilon > 0$ and $n > 0$, we have
\begin{equation*}
    \lambda(\psi^{-1}(B(z,\varepsilon))\cap(\mathbb{R}-[-n,n]))>0 ,
\end{equation*}
where $\lambda$ is the Lebesgue measure on $\mathbb{R}$. The
following proposition from Hoffman's book (see [Hof] pp.171)
relates the local essential range to the values of a function
$f\in A$ in a function algebra $A$ on the fiber $M_{\alpha}(A)$ of
the maximal ideal space of the function algebra:
\begin{proposition} Let $f$ be a function in
$A\subseteq L^{\infty}(\mathbb{T})$ where $A$ is a closed
*-subalgebra of $L^{\infty}(\mathbb{T})$ which contains
$C(\mathbb{T})$. The range of $\hat{f}$ on the fiber
$M_{\alpha}(A)$ consists of all complex numbers $\zeta$ with this
property: for each neighborhood $N$ of $\alpha$ and each
$\varepsilon > 0$, the set
$$\{\mid f-\zeta\mid<\varepsilon\}\cap N$$ has positive Lebesgue
measure. \label{prop4}
\end{proposition}
Hoffman states and proves Proposition 2 for
$A=L^{\infty}(\mathbb{T})$ but in fact his proof works for a
general C*-subalgebra of $L^{\infty}(\mathbb{T})$ that contains
$C(\mathbb{T})$. By Cayley transform Hoffman's proposition holds
for $L^{\infty}(\mathbb{R})$ as well. Let us also recall the
following fact from [Gul] that we will use in the last section:
\begin{lemma}
If $\psi\in$ $QC(\mathbb{R})\cap H^{\infty}(\mathbb{H})$ we have
$$\mathcal{R}_{\infty}(\psi)=\mathcal{C}_{\infty}(\psi)$$ where $\mathcal{C}_{\infty}(\psi)$ is the cluster set of $\psi$ at
infinity which is defined as the set of points $z\in$ $\mathbb{C}$
for which there is a sequence $\{z_{n}\}\subset$ $\mathbb{H}$ so
that $z_{n}\rightarrow$ $\infty$ and $\psi(z_{n})\rightarrow$ $z$.
\end{lemma}
See [Gul] for a proof of this lemma.

 We finish this section by recalling an
elementary geometric lemma from [Gul] which we will use in the
next section:
\begin{lemma}
    Let $K\subset$ $\mathbb{H}$ be a compact
    subset of $\mathbb{H}$. Then there is an $\alpha\in$ $\mathbb{R}^{+}$
    such that $\sup\{\mid\frac{\alpha i-z}{\alpha}\mid:z\in K\}<$
    $\delta < 1$ for some $\delta\in$ $(0,1)$.
    \label{lem1}
\end{lemma}
See [Gul] for a proof of this lemma.

\section{an approximation scheme for quasi-parabolic composition
operators on hardy spaces of the bi-disc}

This section is a generalization of sec. 3 of [Gul] to bi-disc. As
in [Gul] we devise an integral representation formula for
composition operators and we develop an approximation scheme using
this integral formula for composition operators induced by maps of
the form $\varphi:\mathbb{H}^{2}\rightarrow\mathbb{H}^{2}$
$$\varphi(z_{1},z_{2})=(p_{1}z_{1}+\psi_{1}(z_{1},z_{2}),p_{2}z_{2}+\psi(z_{1},z_{2}))$$
where $p_{i}>0$, $i=1,2$ and $\psi_{i}\in
H^{\infty}(\mathbb{H}^{2})$ such that
$\Im(\psi_{i}(z_{1},z_{2}))>\epsilon>0$, $\forall
(z_{1},z_{2})\in\mathbb{H}^{2}$. Boundedness of composition
operators induced by such kind of mappings above is not trivial.
In order to show that quasi-parabolic composition operators on
$H^{2}(\mathbb{H}^{2})$ are bounded we will use the following
lemma due to Jafari (see [Jaf] pp.872):
\begin{lemma}
Suppose $\varphi:\mathbb{H}^{n}\rightarrow\mathbb{H}^{n}$ is
holomorphic and $C_{\varphi}$ is bounded(compact) on a dense
subset of $H^{p}(\mathbb{H}^{n})$ for $1<p<\infty$. Then
$C_{\varphi}$ is bounded(compact)
\end{lemma}
Although Jafari states and proves this lemma for the poly-disc
$\mathbb{D}^{n}$, his proof carries over to our case in exactly
the same manner as he does it for the poly-disc. Using lemma 5 we
prove the following result:
\begin{proposition}
Let
$\varphi(z_{1},z_{2})=(p_{1}z_{1}+\psi_{1}(z_{1},z_{2}),p_{2}z_{2}+\psi_{2}(z_{1},z_{2}))$
be an analytic self-map of $\mathbb{H}^{2}$ into itself such that
$\psi_{j}\in H^{\infty}(\mathbb{H}^{2})$, $p_{j}>0$, and
$\Im(\psi_{j}(z_{1},z_{2}))>\delta>0$ for all
$(z_{1},z_{2})\in\mathbb{H}^{2}$ where $j=1,2$. Then $C_{\varphi}$
is bounded on $H^{2}(\mathbb{H}^{2})$.
\end{proposition}
\begin{proof}
Without loss of generality we may take $p_{1}=p_{2}=1$, for
otherwise we consider the operator $C_{\tilde{\varphi}}$ instead
of $C_{\varphi}$ where
$$\tilde{\varphi}(z_{1},z_{2})=(z_{1}+\psi_{1}(\frac{z_{1}}{p_{1}},\frac{z_{2}}{p_{2}}),z_{2}+\psi_{2}(\frac{z_{1}}{p_{1}},\frac{z_{2}}{p_{2}})).$$
We observe that $V_{p_{1},p_{2}}C_{\tilde{\varphi}}=C_{\varphi}$
where $V_{p_{1},p_{2}}(f)(z_{1},z_{2})=f(p_{1}z_{1},p_{2}z_{2})$.
Since $V_{p_{1},p_{2}}$ is invertible $C_{\varphi}$ is bounded if
and only if $C_{\tilde{\varphi}}$ is bounded.

The Hilbert space $H^{2}(\mathbb{H}^{2})$ is a reproducing kernel
Hilbert space with reproducing kernel functions
$$k_{w_{1},w_{2}}(z_{1},z_{2})=\frac{1}{(2i)^{2}(\overline{w_{1}}-z_{1})(\overline{w_{2}}-z_{2})}$$
we observe that
$$C_{\varphi}^{*}(k_{w_{1},w_{2}})=k_{\varphi(w_{1},w_{2})}$$ and
that
\begin{equation}
\parallel k_{w_{1},w_{2}}\parallel=\frac{1}{\Im(w_{1})\Im(w_{2})}
\end{equation}
 where
$C_{\varphi}^{*}$ is the Hilbert space adjoint of $C_{\varphi}$.
Let
$E=\{\sum_{j=1}^{n}c_{j}k_{w_{1j},w_{2j}}:c_{j}\in\mathbb{C}\}$
then it is clear that $E$ is dense in $H^{2}(\mathbb{H}^{2})$.
Observe that by equation (10) we have
\begin{eqnarray*}
 & &\parallel C_{\varphi}^{*}k_{w_{1},w_{2}}\parallel=\frac{1}{\Im(\varphi_{1}(w_{1},w_{2}))\Im(\varphi_{2}(w_{1},w_{2}))}\\
 & &\leq\left(\frac{\Im(w_{1})}{\Im(w_{1})+\delta}\right)\left(\frac{\Im(w_{2})}{\Im(w_{2})+\delta}\right)\left(\frac{1}{\Im(w_{1})\Im(w_{2})}\right)\leq\parallel
k_{w_{1},w_{2}}\parallel
\end{eqnarray*}
for all $(w_{1},w_{2})\in\mathbb{H}^{2}$, since
$\Im(\psi_{1}(w_{1},w_{2}))>\delta>0$ where
$$\varphi(w_{1},w_{2})=(\varphi_{1}(w_{1},w_{2}),\varphi_{2}(w_{1},w_{2}).$$
Hence $C_{\varphi}^{*}$ is bounded on $E$ and since
$$\mid\langle C_{\varphi}u,v\rangle\mid=\mid\langle
u,C_{\varphi}^{*}v\rangle\mid=\mid\overline{\langle
C_{\varphi}^{*}v,u\rangle}\mid=\mid\langle
C_{\varphi}^{*}v,u\rangle\mid\leq C\parallel u\parallel\parallel
v\parallel$$ for some $C>0$ and for all $u,v\in
H^{2}(\mathbb{H}^{2})$, $C_{\varphi}$ is also bounded on $E$.
Since $E$ is dense in $H^{2}(\mathbb{H}^{2})$, by lemma 5
$C_{\varphi}$ is bounded on $H^{2}(\mathbb{H}^{2})$.
\end{proof}

 Like in one variable
case, for any $f\in H^{2}(\mathbb{H}^{2})$ we have the following
Cauchy Integral formula
\begin{equation}
f(z_{1},z_{2})=\frac{1}{(2\pi
i)^{2}}\int_{-\infty}^{\infty}\int_{-\infty}^{\infty}\frac{f^{*}(x_{1},x_{2})dx_{1}dx_{2}}{(x_{1}-z_{1})(x_{2}-z_{2})}
\end{equation}
Using this integral formula we prove the following proposition
\begin{proposition}
Let $\varphi:\mathbb{H}^{2}\rightarrow\mathbb{H}^{2}$ be an
analytic function such that for
$$\varphi^{*}(\mathbf{x})=\lim_{y\rightarrow 0}\varphi(\mathbf{x}+i\mathbf{y})$$
where $\mathbf{y}=(y,y)\in\mathbb{R}^{2}$ and
$\varphi^{*}(x_{1},x_{2})=(\varphi_{1}(x_{1},x_{2}),\varphi_{2}(x_{1},x_{2}))$
we have $\Im(\varphi_{i}^{*}(x_{1},x_{2}))>0$ for almost
every\quad $\mathbf{x}=(x_{1},x_{2})\in\mathbb{R}^{2}$. Then the
composition operator $C_{\varphi}$ on $H^{2}(\mathbb{H}^{2})$ is
given by
\begin{eqnarray*}
& &(C_{\varphi}f)^{*}(x_{1},x_{2})=\lim_{y\rightarrow
0}(C_{\varphi})(\mathbf{x}+i\mathbf{y})\\
=& &\frac{1}{(2\pi
i)^{2}}\int_{-\infty}^{\infty}\int_{-\infty}^{\infty}\frac{f^{*}(t_{1},t_{2})dt_{1}dt_{2}}{(t_{1}-\varphi_{1}^{*}(\mathbf{x}))(t_{2}-\varphi_{2}^{*}(\mathbf{x}))}.
\end{eqnarray*}
\end{proposition}
\begin{proof} By the equation (11) above one has
\begin{equation*}
(C_{\varphi}f)(\mathbf{x}+i\mathbf{y})=\frac{1}{(2\pi
i)^{2}}\int_{-\infty}^{\infty}\int_{-\infty}^{\infty}\frac{f^{*}(t_{1},t_{2})dt_{1}dt_{2}}{(t_{1}-\varphi_{1}(\mathbf{x}+i\mathbf{y}))(t_{2}-\varphi_{2}(\mathbf{x}+i\mathbf{y}))}
\end{equation*}
where $\mathbf{x}=(x_{1},x_{2})\in\mathbb{R}^{2}$ and
$\mathbf{y}=(y,y)\in\mathbb{R}^{2}$. Let
$\mathbf{x}=(x_{1},x_{2})\in\mathbb{R}^{2}$ be such that
$\lim_{y\rightarrow
    0}\varphi(\mathbf{x}+i\mathbf{y})=$ $\varphi^{*}(\mathbf{x})=$ $(\varphi_{1}^{*}(\mathbf{x}),\varphi_{2}^{*}(\mathbf{x}))$ exists and $\Im(\varphi_{j}^{*}(\mathbf{x}))>
    0$, $j=1,2$. We have
\begin{eqnarray}
        & &\left\lvert C_{\varphi}(f)(\mathbf{x}+i\mathbf{y})-\frac{1}{(2\pi i)^{2}}\int_{-\infty}^{\infty}\int_{-\infty}^{\infty}\frac{f^{*}(t_{1},t_{2})dt_{1}dt_{2}}{(t_{1}-\varphi_{1}^{*}(\mathbf{x}))(t_{2}-\varphi_{2}^{*}(\mathbf{x}))}\right\rvert\\
        =&\nonumber &\left\lvert\frac{1}{(2\pi i)^{2}}\bigg(\int_{\mathbb{R}^{2}}\frac{f^{*}(t_{1},t_{2})dt_{1}dt_{2}}{\prod_{j=1}^{2}(t_{j}-\varphi_{j}(\mathbf{x}+i\mathbf{y}))}-\int_{\mathbb{R}^{2}}\frac{f^{*}(t_{1},t_{2})dt_{1}dt_{2}}{\prod_{j=1}^{2}(t_{j}-\varphi_{j}^{*}(\mathbf{x}))}\bigg)\right\rvert\\
        =&\nonumber &\left\lvert\frac{1}{(2\pi i)^{2}}\int_{\mathbb{R}^{2}}\bigg(\frac{1}{\prod_{j=1}^{2}(t_{j}-\varphi_{j}(\mathbf{x}+i\mathbf{y}))}-\frac{1}{\prod_{j=1}^{2}(t_{j}-\varphi_{j}^{*}(\mathbf{x}))}\bigg)f^{*}(\mathbf{t})d\mathbf{t}\right\rvert
\end{eqnarray}
Consider
\begin{eqnarray*}
       & &\frac{1}{(t_{1}-\varphi_{1}(\mathbf{x}+i\mathbf{y}))(t_{2}-\varphi_{2}(\mathbf{x}+i\mathbf{y}))}-\frac{1}{(t_{1}-\varphi_{1}^{*}(\mathbf{x}))(t_{2}-\varphi_{2}^{*}(\mathbf{x}))}\\
       =& &\frac{1}{(t_{1}-\varphi_{1}(\mathbf{x}+i\mathbf{y}))(t_{2}-\varphi_{2}(\mathbf{x}+i\mathbf{y}))}-\frac{1}{(t_{1}-\varphi_{1}^{*}(\mathbf{x}+i\mathbf{y}))(t_{2}-\varphi_{2}^{*}(\mathbf{x}))}\\
       +& &\frac{1}{(t_{1}-\varphi_{1}(\mathbf{x}+i\mathbf{y}))(t_{2}-\varphi_{2}(\mathbf{x}))}-\frac{1}{(t_{1}-\varphi_{1}^{*}(\mathbf{x}))(t_{2}-\varphi_{2}^{*}(\mathbf{x}))}\\
       =& &\frac{\varphi_{2}(x_{1}+iy,x_{2}+iy)-\varphi_{2}^{*}(x_{1},x_{2})}{(t_{1}-\varphi_{1}(\mathbf{x}+i\mathbf{y}))(t_{2}-\varphi_{2}(\mathbf{x}))(t_{2}-\varphi_{2}(\mathbf{x}+i\mathbf{y}))}\\
       +& &\frac{\varphi_{1}(x_{1}+iy,x_{2}+iy)-\varphi_{1}^{*}(x_{1},x_{2})}{(t_{1}-\varphi_{1}(\mathbf{x}+i\mathbf{y}))(t_{2}-\varphi_{2}(\mathbf{x}))(t_{2}-\varphi_{2}(\mathbf{x}+i\mathbf{y}))}
\end{eqnarray*}
Inserting this into equation (12) above we obtain
\begin{eqnarray*}
  & &\left\lvert C_{\varphi}(f)(\mathbf{x}+i\mathbf{y})-\frac{1}{(2\pi i)^{2}}\int_{-\infty}^{\infty}\int_{-\infty}^{\infty}\frac{f^{*}(t_{1},t_{2})dt_{1}dt_{2}}{(t_{1}-\varphi_{1}^{*}(\mathbf{x}))(t_{2}-\varphi_{2}^{*}(\mathbf{x}))}\right\rvert\\
  \leq& &\frac{\mid\varphi_{2}(\mathbf{x}+i\mathbf{y})-\varphi_{2}^{*}(\mathbf{x})\mid}{4\pi^{2}}\left\lvert\int_{\mathbb{R}^{2}}\frac{f^{*}(t_{1},t_{2})dt_{1}dt_{2}}{(t_{1}-\varphi_{1}(\mathbf{x}+i\mathbf{y}))(t_{2}-\varphi_{2}(\mathbf{x}))(t_{2}-\varphi_{2}(\mathbf{x}+i\mathbf{y}))}\right\rvert\\
  +& &\frac{\mid\varphi_{1}(\mathbf{x}+i\mathbf{y})-\varphi_{1}^{*}(\mathbf{x})\mid}{4\pi^{2}}\left\lvert\int_{\mathbb{R}^{2}}\frac{f^{*}(t_{1},t_{2})dt_{1}dt_{2}}{(t_{1}-\varphi_{1}(\mathbf{x}+i\mathbf{y}))(t_{2}-\varphi_{2}(\mathbf{x}))(t_{2}-\varphi_{2}(\mathbf{x}+i\mathbf{y}))}\right\rvert\\
  \leq& &\parallel f\parallel_{2}\frac{\mid\varphi_{2}(\mathbf{x}+i\mathbf{y})-\varphi_{2}^{*}(\mathbf{x})\mid}{4\pi^{2}}\bigg(\int_{\mathbb{R}^{2}}\frac{dt_{1}dt_{2}}{\mid(t_{2}-\varphi_{2}(\mathbf{x}))\prod_{j=1}^{2}(t_{j}-\varphi_{j}(\mathbf{x}+i\mathbf{y}))\mid^{2}}\bigg)^{\frac{1}{2}}\\
  +& &\parallel f\parallel_{2}\frac{\mid\varphi_{1}(\mathbf{x}+i\mathbf{y})-\varphi_{1}^{*}(\mathbf{x})\mid}{4\pi^{2}}\bigg(\int_{\mathbb{R}^{2}}\frac{dt_{1}dt_{2}}{\mid(t_{2}-\varphi_{2}(\mathbf{x}))\prod_{j=1}^{2}(t_{j}-\varphi_{j}(\mathbf{x}+i\mathbf{y}))\mid^{2}}\bigg)^{\frac{1}{2}}
\end{eqnarray*}
by Cauchy-Schwarz inequality. Now consider
\begin{eqnarray*}
 & &\int_{\mathbb{R}^{2}}\frac{dt_{1}dt_{2}}{\mid(t_{1}-\varphi_{1}(\mathbf{x}+i\mathbf{y}))(t_{2}-\varphi_{2}(\mathbf{x}))(t_{2}-\varphi_{2}(\mathbf{x}+i\mathbf{y}))\mid^{2}}\\
 =& &\bigg(\int_{-\infty}^{\infty}\frac{dt_{1}}{\mid t_{1}-\varphi_{1}(\mathbf{x}+i\mathbf{y})\mid^{2}}\bigg)\bigg(\int_{-\infty}^{\infty}\frac{dt_{1}}{\mid(t_{2}-\varphi_{2}(\mathbf{x}))(t_{2}-\varphi_{2}(\mathbf{x}+i\mathbf{y}))\mid^{2}}\bigg)
\end{eqnarray*}
We have $\lim_{y\rightarrow
0}\varphi_{j}(\mathbf{x}+i\mathbf{y})=\varphi_{j}^{*}(\mathbf{x})$,
$j=1,2$. Let
\begin{equation*}
\varepsilon_{0}=\frac{\inf\{\mid
t-\varphi_{1}^{*}(\mathbf{x}):t\in\mathbb{R}\}\cup\{\mid
t-\varphi_{2}^{*}(\mathbf{x}):t\in\mathbb{R}\}}{2}
\end{equation*}
Choose $\varepsilon_{0}>\varepsilon>0$ such that $\forall$
$0<y<\delta$ we have
\begin{equation*}
\mid\varphi_{j}(\mathbf{x}+i\mathbf{y})-\varphi_{j}(\mathbf{x})\mid<\frac{\varepsilon}{2}
\end{equation*}
Then one has by triangle inequality
\begin{equation*}
\mid t_{1}-\varphi_{1}(\mathbf{x}+i\mathbf{y})\mid\geq\mid
t_{1}-\varphi_{1}^{*}(\mathbf{x})\mid-\varepsilon_{0}>\varepsilon_{0}
\end{equation*}
and this implies that
\begin{equation*}
\frac{1}{\mid
t_{1}-\varphi_{1}(\mathbf{x}+i\mathbf{y})\mid}\leq\frac{1}{\mid
t_{1}-\varphi_{1}^{*}(\mathbf{x})\mid-\varepsilon_{0}}
\end{equation*}
hence we have
\begin{equation*}
\int_{-\infty}^{\infty}\frac{dt_{1}}{\mid
t_{1}-\varphi_{1}(\mathbf{x}+i\mathbf{y})\mid^{2}}\leq\int_{-\infty}^{\infty}\frac{dt_{1}}{(\mid
t_{1}-\varphi_{1}^{*}(\mathbf{x})\mid-\varepsilon_{0})^{2}}=M_{\varepsilon_{0},\mathbf{x}}
\end{equation*}
since the integral on the right hand side converges and its value
only depends on $\varepsilon_{0}$ and
$\mathbf{x}=(x_{1},x_{2})\in\mathbb{R}^{2}$. Similarly by the same
arguments as in [Gul] we have
\begin{eqnarray*}
   & &\int_{-\infty}^{\infty}\frac{dt_{2}}{\mid(t_{2}-\varphi_{2}^{*}(\mathbf{x}))(t_{2}-\varphi_{2}(\mathbf{x}+i\mathbf{y}))\mid^{2}}\\
   \leq& &\int_{-\infty}^{\infty}\frac{dt_{2}}{(\mid t_{2}-\varphi_{2}^{*}(\mathbf{x})\mid-\varepsilon_{0})^{2}\mid t_{2}-\varphi_{2}^{*}(\mathbf{x})\mid^{2}}=K_{\varepsilon_{0},\mathbf{x}}
\end{eqnarray*}
As a result we have
\begin{eqnarray*}
  & &\left\lvert C_{\varphi}(f)(\mathbf{x}+i\mathbf{y})-\frac{1}{(2\pi i)^{2}}\int_{-\infty}^{\infty}\int_{-\infty}^{\infty}\frac{f^{*}(t_{1},t_{2})dt_{1}dt_{2}}{(t_{1}-\varphi_{1}^{*}(\mathbf{x}))(t_{2}-\varphi_{2}^{*}(\mathbf{x}))}\right\rvert\\
  \leq& &\frac{(M_{\varepsilon_{0},\mathbf{x}}K_{\varepsilon_{0},\mathbf{x}})^{\frac{1}{2}}}{4\pi^{2}}(\mid\varphi_{1}(\mathbf{x}+i\mathbf{y})-\varphi_{1}^{*}(\mathbf{x})\mid+\mid\varphi_{2}(\mathbf{x}+i\mathbf{y})-\varphi_{2}^{*}(\mathbf{x})\mid)\\
  \leq& &\frac{(M_{\varepsilon_{0},\mathbf{x}}K_{\varepsilon_{0},\mathbf{x}})^{\frac{1}{2}}}{4\pi^{2}}\varepsilon
\end{eqnarray*}
Therefore
\begin{equation*}
\lim_{y\rightarrow
0}(C_{\varphi}f)(x_{1}+iy,x_{2}+iy)=\frac{1}{(2\pi
i)^{2}}\int_{-\infty}^{\infty}\int_{-\infty}^{\infty}\frac{f^{*}(t_{1},t_{2})dt_{1}dt_{2}}{(t_{1}-\varphi_{1}^{*}(x_{1},x_{2}))(t_{2}-\varphi_{2}^{*}(x_{1},x_{2}))}
\end{equation*}
\end{proof}
Throughout the rest of the paper we will identify a function $f$
in $H^{2}(\mathbb{H}^{2})$ or $H^{\infty}(\mathbb{H}^{2})$ with
its boundary function $f^{*}$. We formulate and prove our
approximation scheme as the following proposition.

\begin{proposition}
Let $\varphi:\mathbb{H}^{2}\rightarrow$ $\mathbb{H}^{2}$ be an
analytic self-map of $\mathbb{H}^{2}$ such that
$$\varphi(z_{1},z_{2})=(p_{1}z_{1}+\psi_{1}(z_{1},z_{2}),p_{2}z_{2}+\psi_{2}(z_{1},z_{2}))$$ $p_{1},p_{2} > 0$ and
    $\psi_{j}\in$ $H^{\infty}$ is such that $\Im(\psi_{j}(z))>\epsilon>0$, $j=1,2$ for
    all $(z_{1},z_{2})\in\mathbb{H}^{2}$. Then there is an $\alpha\in$ $\mathbb{R}^{+}$
    such that for $C_{\varphi}:H^{2}(\mathbb{H}^{2})\rightarrow$ $H^{2}(\mathbb{H}^{2})$ we have
\begin{equation*}
    C_{\varphi} = V_{p_{1},p_{2}}\sum_{n,m=0}^{\infty}T_{\tau_{1}^{n}}T_{\tau_{2}^{m}}D_{\vartheta_{1,n}}D_{\vartheta_{2,m}},
\end{equation*}
    where the convergence of the series is in operator norm,
    $T_{\tau_{1}^{n}}$ and $T_{\tau_{2}^{m}}$ are the Toeplitz operators with symbols $\tau_{1}^{n}$ and $\tau_{2}^{m}$ respectively,
\begin{equation*}
    \tau_{j}(x_{1},x_{2})=i\alpha-\tilde{\psi_{j}}(x_{1},x_{2}),\quad
    \tilde{\psi}(x_{1},x_{2})=\psi(\frac{x_{1}}{p_{1}},\frac{x_{2}}{p_{2}}),
\end{equation*}
     $V_{p_{1},p_{2}}$ is the dilation
    operator defined as
\begin{equation*}
(V_{p_{1},p_{2}}f)(z_{1},z_{2})=f(\frac{z_{1}}{p_{1}},\frac{z_{2}}{p_{2}})
\end{equation*}

     and $D_{\vartheta_{1,n}}$ and $D_{\vartheta_{2,m}}$ are the
    Fourier multipliers with $\vartheta_{1,n}(t_{1},t_{2})=$
    $\frac{(-it_{1})^{n}e^{-\alpha t_{1}}}{n!}$ and $\vartheta_{2,m}(t_{1},t_{2})=$
    $\frac{(-it_{2})^{m}e^{-\alpha t_{2}}}{m!}$ respectively. \label{prop3}
\end{proposition}
\begin{proof}
    Since for $\varphi(z_{1},z_{2})=$ $(p_{1}z_{1}+\psi_{1}(z_{1},z_{2}),p_{2}z_{2}+\psi_{2}(z_{1},z_{2}))$ where $\psi_{j}\in$
    $H^{\infty}$ with $\Im(\psi_{j}(z_{1},z_{2})) > \epsilon > 0$ for all $z\in$
    $\mathbb{H}$ and $p_{1},p_{2} > 0$, we have
\begin{equation*}
    \Im(\varphi_{j}^{*}(x_{1},x_{2}))\geq\epsilon>0\quad\textrm{for almost every}\quad (x_{1},x_{2})\in\mathbb{R}^{2}.
\end{equation*}
    We can use Proposition 7
    for $C_{\varphi}:H^{2}\rightarrow$ $H^{2}$ to have
    \begin{eqnarray*}
    & &(C_{\varphi}f)(x_{1},x_{2}) = \frac{1}{(2\pi i)^{2}}\int_{\mathbb{R}^{2}}\frac{f(w_{1},w_{2})dw_{1}dw_{2}}{(w_{1}-\varphi_{1}(\mathbf{x}))(w_{2}-\varphi_{2}(\mathbf{x}))}\\
    =& &\frac{1}{(2\pi i)^{2}}\int_{\mathbb{R}^{2}}\frac{f(w_{1},w_{2})dw_{1}dw_{2}}{(w_{1}-px_{1}-\psi_{1}(\mathbf{x}))(w_{2}-px_{2}-\psi_{2}(\mathbf{x}))}.
    \end{eqnarray*}
    where $\mathbf{x}=(x_{1},x_{2})\in\mathbb{R}^{2}$. Without loss of generality, we take $p_{1}=p_{2}=1$, since if $p_{1}\neq 1$ or $p_{2}\neq 1$ then we have
    \begin{equation}
        (V_{\frac{1}{p_{1}},\frac{1}{p_{2}}}C_{\varphi})(f)(x_{1},x_{2}) = \frac{1}{(2\pi i)^{2}}\int_{\mathbb{R}^{2}}\frac{f(w_{1},w_{2})dw_{1}dw_{2}}{\prod_{j=1}^{2}(w_{j}-x_{j}-\tilde{\psi_{j}}(x_{1},x_{2}))},
    \end{equation}
     where $\tilde{\psi_{j}}(x_{1},x_{2})=$ $\psi_{j}(\frac{x_{1}}{p_{1}},\frac{x_{2}}{p_{2}})$ and
    $V_{\beta_{1},\beta_{2}}f(z_{1},z_{2})=$ $f(\beta_{1}z_{1},\beta_{2}z_{2})$ ($\beta_{1},\beta_{2} > 0$) is the dilation
    operator. We observe that
    \begin{eqnarray}
        & &\frac{1}{(w_{1}-x_{1}-\psi_{1}(\mathbf{x}))(w_{2}-x_{2}-\psi_{2}(\mathbf{x}))} =\\
        &\nonumber &\frac{1}{(x_{1}-w_{1}+i\alpha-(i\alpha-\psi_{1}(\mathbf{x})))(x_{2}-w_{2}+i\alpha-(i\alpha-\psi_{2}(\mathbf{x})))}= \\
        &\nonumber &\frac{1}{(x_{1}-w_{1}+i\alpha)(x_{2}-w_{2}+i\alpha)\left(1-\left(\dfrac{i\alpha-\psi_{1}(\mathbf{x})}{x_{1}-w_{1}+i\alpha}\right)\right)\left(1-\left(\dfrac{i\alpha-\psi_{2}(\mathbf{x})}{x_{2}-w_{2}+i\alpha}\right)\right)} .
    \end{eqnarray}
    Since $\Im(\psi_{j}(z_{1},z_{2}))>\epsilon>0$ for all $(z_{1},z_{2})\in\mathbb{H}^{2}$ and
    $\psi_{j}\in H^{\infty}$, we have $\overline{\psi_{j}(\mathbb{H}^{2})}$ is
    compact in $\mathbb{H}$, and then by Lemma 4 there is an $\alpha > 0$ such that
\begin{equation*}
    \left\lvert\frac{i\alpha-\psi_{j}(\mathbf{x})}{x_{j}-w_{j}+i\alpha}\right\rvert<\delta<1
\end{equation*}
    for all $(x_{1},x_{2}),(w_{1},w_{2})\in\mathbb{R}^{2}$. So we have
    \begin{equation*}
        \frac{1}{1-\left(\dfrac{i\alpha-\psi_{j}(\mathbf{x})}{x_{j}-w_{j}+i\alpha}\right)} =
        \sum_{n=0}^{\infty}\left(\frac{i\alpha-\psi_{j}(\mathbf{x})}{x_{j}-w_{j}+i\alpha}\right)^{n}.
    \end{equation*}
    Inserting this into equation (14) and then into equation (13), we
    have
    \begin{equation*}
        (C_{\varphi}f)(x_{1},x_{2}) = \sum_{n,m=0}^{(N_{1}-1),(N_{2}-1)}T_{\tau_{1}^{n}}T_{\tau_{2}^{m}}K_{n,m}f(\mathbf{x}) +
        R_{1,N_{1}}f(\mathbf{x})+ R_{2,N_{2}}f(\mathbf{x})+R_{N_{1},N_{2}}f(\mathbf{x}),
    \end{equation*}
    where $T_{\tau_{j}^{n}}f(x_{1},x_{2})=$ $\tau_{j}^{n}(x_{1},x_{2})f(x_{1},x_{2})$, $\tau_{j}(x_{1},x_{2})=$ $i\alpha
    -\psi_{j}(x_{1},x_{2})$, $K_{n,m}$ is defined as
\begin{equation*}
K_{n,m}f(x_{1},x_{2})=\frac{1}{(2\pi
i)^{2}}\int_{\mathbb{R}^{2}}\frac{f(w_{1},w_{2})dw_{1}dw_{2}}{(x_{1}-w_{1}+i\alpha)^{n+1}(x_{2}-w_{2}+i\alpha)^{m+1}}
\end{equation*}
and
\begin{equation*}
R_{1,N_{1}}f(x_{1},x_{2})=\sum_{m=0}^{N_{2}}T_{\tau_{1}^{N_{1}+1}}T_{\tau_{2}^{m}}K_{N_{1}+1,m}f(\mathbf{x}),
\end{equation*}
\begin{equation*}
R_{2,N_{2}}f(x_{1},x_{2})=\sum_{n=0}^{N_{1}}T_{\tau_{1}^{n}}T_{\tau_{2}^{N_{2}+1}}K_{n,N_{2}+1}f(\mathbf{x}),
\end{equation*}
\begin{eqnarray*}
    & &R_{N_{1},N_{2}}f(x)=\\
    & &\frac{1}{(2\pi i)^{2}}T_{\tau_{1}^{N_{1}+1}}T_{\tau_{2}^{N_{2}+1}}\int_{\mathbb{R}^{2}}\frac{f(w_{1},w_{2})dw_{1}dw_{2}}{\prod_{j=1}^{2}(x_{j}-w_{j}+i\alpha)^{N_{j}+1}(w_{j}-x_{j}-\psi_{j}(\mathbf{x}))}.
\end{eqnarray*}
Let
$\varphi_{1}(z_{1},z_{2})=(z_{1}+\psi_{1}(z_{1},z_{2}),z_{2}+i\alpha)$
and
$\varphi_{2}(z_{1},z_{2})=(z_{1}+i\alpha,z_{2}+\psi_{2}(z_{1},z_{2}))$
then we have the following estimates for $R_{1,N_{1}}$ and
$R_{2,N_{2}}$:
\begin{equation*}
\parallel R_{1,N_{1}}\parallel\leq\parallel C_{\varphi_{1}}\parallel\delta^{N_{1}+1}(1-\delta)^{-1},
\end{equation*}
\begin{equation*}
\parallel R_{2,N_{2}}\parallel\leq\parallel C_{\varphi_{2}}\parallel\delta^{N_{2}+1}(1-\delta)^{-1}.
\end{equation*}
By proposition 6, $C_{\varphi_{1}}$ and $C_{\varphi_{2}}$ are
bounded so we have $\parallel R_{1,N_{1}}\parallel\rightarrow 0$
and $\parallel R_{2,N_{2}}\parallel\rightarrow 0$ as
$N_{1},N_{2}\rightarrow\infty$. We have the following estimate for
$R_{N_{1},N_{2}}$:
\begin{equation*}
\parallel R_{N_{1},N_{2}}\parallel\leq\parallel
T_{\tau_{1}}\parallel\parallel T_{\tau_{2}}\parallel\parallel
C_{\varphi}\parallel\delta^{N_{1}+N_{2}}
\end{equation*} Hence $\parallel R_{N_{1},N_{2}}\parallel\rightarrow 0$ as
$N_{1},N_{2}\rightarrow\infty$. We observe that
\begin{equation*}
K_{n,m}=D_{\vartheta_{n,m}}
\end{equation*}
where $\vartheta_{n,m}(t_{1},t_{2})=\frac{(-it_{1})^{n}e^{-\alpha
t_{1}}}{n!}\frac{(-it_{2})^{m}e^{-\alpha t_{2}}}{m!}$. Hence we have
\begin{equation*}
    C_{\varphi} = \sum_{n,m=0}^{\infty}T_{\tau_{1}^{n}}T_{\tau_{2}^{m}}D_{\vartheta_{1,n}}D_{\vartheta_{2,m}},
\end{equation*}
where the convergence is in operator norm.
\end{proof}

\section{a $\Psi$-c*-algebra of operators on hardy spaces of $\mathbb{H}^{2}$}

In the preceding section we have shown that ``quasi-parabolic''
composition operators on the $\mathbb{H}^{2}$ lie in the
C*-algebra generated by certain Toeplitz operators and Fourier
multipliers. In this section we will identify the character space
of the C*-algebra generated by Toeplitz operators with a class of
symbols and Fourier multipliers. We identify this C*-algebra with
the tensor product of its one variable version which is treated in
[Gul] with itself. We will consider the C*-algebra of operators
acting on $H^{2}(\mathbb{H}^{2})$
\begin{equation*}
\Psi(QC(\mathbb{R}^{2}),C_{0}((\mathbb{R}^{+})^{2}))=C^{*}(\mathcal{T}(QC(\mathbb{R}^{2}))\cup
F_{C_{0}((\mathbb{R}^{+})^{2})})
\end{equation*}
where $$QC(\mathbb{R}^{2})=QC(\mathbb{R})\otimes
QC(\mathbb{R})=\overline{\{\sum_{j=1}^{n}f_{j}(x).g_{j}(y):f_{j},g_{j}\in
QC(\mathbb{R})\}}$$ and
$$\mathcal{T}(QC(\mathbb{R}^{2}))=C^{*}(\{T_{\phi}:\phi\in
QC(\mathbb{R}^{2})\}$$ is the Toeplitz C*-algebra with
$QC(\mathbb{R}^{2})$ symbols. Recall that $$QC(\mathbb{R})=\{f\in
L^{\infty}(\mathbb{R}):f\circ\mathfrak{C}^{-1}\in QC(\mathbb{T})\}$$
where $$\mathfrak{C}(z)=\frac{z-i}{z+i}$$ is the Cayley transform
and
$$QC(\mathbb{T})=(H^{\infty}(\mathbb{D})+C(\mathbb{T}))\cap\overline{H^{\infty}(\mathbb{D})+C(\mathbb{T})}$$
is the class of quasi-continuous functions.

In ([Gul]) we showed that the following sequence
\begin{equation}
0\xrightarrow{} K(H^{2}(\mathbb{H}))\xrightarrow{j}
\Psi(QC(\mathbb{R}),C_{0}(\mathbb{R}^{+}))\xrightarrow{\pi}
C(\mathbb{M})\xrightarrow{} 0
\end{equation}
is short exact where
$$\Psi(QC(\mathbb{R}),C_{0}(\mathbb{R}^{+}))=C^{*}(\mathcal{T}(QC(\mathbb{R}))\cup
F_{C_{0}(\mathbb{R}^{+})})$$ is the C*-algebra generated by
Toeplitz operators with QC symbols and continuous Fourier
multipliers and
\begin{equation}
\mathbb{M}\cong (M_{\infty}(QC(\mathbb{R}))\times
[0,\infty])\cup(M(QC(\mathbb{R}))\times\{\infty\})
\end{equation}
is the maximal ideal space of
$\Psi(QC(\mathbb{R}),C_{0}(\mathbb{R}^{+}))/K(H^{2}(\mathbb{H}))$.
Here $M(QC(\mathbb{R}))$ is the maximal ideal space of
$QC(\mathbb{R})$ and $$M_{\infty}(QC(\mathbb{R}))=\{x\in
M(QC(\mathbb{R})):x|_{C(\dot{\mathbb{R}})}=\delta_{\infty},\quad
\delta_{\infty}(f)=\lim_{t\rightarrow\infty}f(t)\}$$ is the fiber
of $M(QC(\mathbb{R})$ at $\infty$. Throughout the C*-algebra
$\Psi(QC(\mathbb{R}),C_{0}(\mathbb{R}^{+}))$ will be denoted by
$\Psi$ and $\Psi(QC(\mathbb{R}^{2}),C_{0}((\mathbb{R}^{+})^{2}))$
will be denoted by $\Psi_{2}$. Since $K(H^{2})$ is nuclear and
$C(\mathbb{M})$ is commutative and hence nuclear, by eqn. (15)
$\Psi$ is nuclear. Therefore all the C*-algebras that we will deal
with in this paper will be nuclear and $A\otimes B$ will denote
the closure of the algebraic tensor product of $A$ and $B$ with
respect to this unique C* norm. Following the approach in [DoH] we
identify $\Psi_{2}$ with $\Psi\otimes\Psi$ corresponding to the
identification of $H^{2}(\mathbb{H}^{2})$ with
$H^{2}(\mathbb{H})\otimes H^{2}(\mathbb{H})$. Define the operators
$W_{f}=T_{f}\hat{\otimes}I$ for any $f\in QC(\mathbb{R})$ as
\begin{equation*}
(W_{f}a)(z_{1},z_{2})=P(f(z_{1})a(z_{1},z_{2}))
\end{equation*}
 for $a\in H^{2}(\mathbb{H}^{2})$ and $W_{g}=I\hat{\otimes} T_{g}$ as
\begin{equation*}
(W_{g}a)(z_{1},z_{2})=P(g(z_{2})a(z_{1},z_{2}))
\end{equation*}
where $P$ is the orthogonal projection of $L^{2}(\mathbb{R}^{2})$
onto $H^{2}(\mathbb{H}^{2})$. In the same way for the Fourier
multipliers define $E_{\vartheta}=D_{\vartheta}\hat{\otimes} I$ as
\begin{equation*}
(E_{\vartheta}a)(z_{1},z_{2})=(\mathcal{F}^{-1}M_{\vartheta}\mathcal{F})(a)(z_{1},z_{2})
\end{equation*}
where $\vartheta\in$ $C_{0}(\mathbb{R}^{+})$ and
$E_{\tau}=I\hat{\otimes} D_{\tau}$ for $\tau\in$
$C_{0}(\mathbb{R}^{+})$ as
\begin{equation*}
(E_{\tau}a)(z_{1},z_{2})=(\mathcal{F}^{-1}M_{\tau}\mathcal{F})(a)(z_{1},z_{2})
\end{equation*}
where $M_{\vartheta}$ is defined as
\begin{equation*}
(M_{\vartheta}a)(t_{1},t_{2})=\vartheta(t_{1})a(t_{1},t_{2})
\end{equation*}
$M_{\tau}$ is defined as
\begin{equation*}
(M_{\tau}a)(t_{1},t_{2})=\tau(t_{2})a(t_{1},t_{2})
\end{equation*}
and $\mathcal{F}$ is the Fourier transform defined as in equation
(2). Since $\Psi_{2}$ is generated by
$\{W_{f},W_{g},E_{\vartheta},E_{\tau}:f,g\in
QC(\mathbb{R})\quad\vartheta , \tau\in C_{0}(\mathbb{R}^{+})\}$
and $\Psi$ is nuclear, $\Psi_{2}=\Psi\otimes\Psi$. By equation (6)
we have $K(H^{2}(\mathbb{H}^{2}))=K(H^{2}(\mathbb{H}))\otimes
K(H^{2}(\mathbb{H}))$. Since $\Psi_{2}=\Psi\otimes\Psi$ we have
\begin{equation*}
com(\Psi_{2})=com(\Psi\otimes\Psi)=I^{*}(com(\Psi)\otimes\Psi\cup\Psi\otimes
com(\Psi))
\end{equation*}
By equation (15) we have $com(\Psi)=K(H^{2}(\mathbb{H}))$ hence we
have
\begin{equation*}
K(H^{2}(\mathbb{H}^{2}))=K(H^{2}(\mathbb{H}))\otimes
K(H^{2}(\mathbb{H}))\subset com(\Psi\otimes\Psi)=com(\Psi_{2})
\end{equation*}
Hence by equations (3) and (7) we have
\begin{equation*}
M(\Psi_{2}/K(H^{2}(\mathbb{H}^{2})))=M(\Psi_{2})=M(\Psi\otimes\Psi)\cong\mathbb{M}\times\mathbb{M}
\end{equation*}
where $\mathbb{M}$ is as in equation (16). We summarize the result
of this section as the following proposition:
\begin{proposition}
Let $\Psi_{2}=C^{*}(\mathcal{T}(QC(\mathbb{R}^{2}))\cup
F_{C_{0}((\mathbb{R}^{+})^{2})})$ be the C*-algebra generated by
Toeplitz operators with $QC(\mathbb{R})\otimes QC(\mathbb{R})$
symbols and continuous Fourier multipliers acting on
$H^{2}(\mathbb{H}^{2})$. Then for the character space
$M(\Psi_{2}/K(H^{2}(\mathbb{H}^{2})))$ of
$\Psi_{2}/K(H^{2}(\mathbb{H}^{2}))$ we have
\begin{equation*}
M(\Psi_{2}/K(H^{2}(\mathbb{H}^{2})))\cong\mathbb{M}\times\mathbb{M}
\end{equation*}
where $\mathbb{M}$ is the maximal ideal space of the C*-algebra
$\Psi(QC(\mathbb{R}),C_{0}(\mathbb{R}^{+}))/K(H^{2}(\mathbb{H}))$
generated by Toeplitz operators with $QC(\mathbb{R})$ symbols and
continuous Fourier multipliers modulo compact operators acting on
$H^{2}(\mathbb{H})$.
\end{proposition}

\section{main results}

In this section we prove the main results of this paper which
asserts that the essential spectra of quasi-parabolic composition
operators on the Hardy space of the poly-disc contain a
non-trivial set which consists of spiral curves as in one variable
case. In doing this we use multi-dimensional generalizations of
the methods employed in [Gul]. We prove the following proposition
which might be regarded as a weaker version of a multi-dimensional
generalization of lemma 3:
\begin{proposition}
Let $\psi\in QC(\mathbb{R})\otimes QC(\mathbb{R})\cap
H^{\infty}(\mathbb{H}^{2})$ then we have
\begin{equation*}
\{(\phi_{1}\hat{\otimes}\phi_{2})(\psi):\phi_{1},\phi_{2}\in
M_{\infty}(QC(\mathbb{R}))\}\supseteq\mathcal{C}_{(\infty,\infty)}(\psi)
\end{equation*}
where $\phi_{1}\hat{\otimes}\phi_{2}$
  is as defined by equation (8) and $\mathcal{C}_{(\infty,\infty)}(\psi)$ is defined to be the
set of points
  $w\in$ $\mathbb{C}$ for which there is a sequence $\{z_{n}\}\subset$
  $\mathbb{H}^{2}$ so that $z_{n}\rightarrow$ $(\infty,\infty)$ and
  $\psi(z_{n})\rightarrow$ $w$.
\end{proposition}
\begin{proof}
Let us first show the above inclusion for functions of the form
\begin{equation*}
\psi(z_{1},z_{2})=\sum_{j=1}^{m}\varphi_{j}(z_{1})\eta_{j}(z_{2})
\end{equation*}
where $\varphi_{j},\eta_{j}\in QC(\mathbb{R})\cap
H^{\infty}(\mathbb{H})$. Let
$w\in\mathcal{C}_{(\infty,\infty)}(\psi)$ then there exists a
sequence $\{z_{n}\}\in\mathbb{H}^{2}$, $z_{n}=(z_{1,n},z_{2,n})$
such that $z_{n}=(z_{1,n},z_{2,n})\rightarrow (\infty,\infty)$ and
$\psi(z_{1,n},z_{2,n})\rightarrow w$. Since
$\varphi_{j},\eta_{j}\in QC(\mathbb{R})\cap
H^{\infty}(\mathbb{H})\subset H^{\infty}(\mathbb{H})$, the
sequences $\{\varphi_{j}(z_{1,n})\}$ and $\{\eta_{j}(z_{2,n})\}$
have convergent subsequences, hence without loss of generality (by
passing to a subsequence if needed) there are $w_{1,j},
w_{2,j}\in\mathbb{H}$ such that
\begin{equation}
\varphi_{j}(z_{1,n})\rightarrow
w_{1,j}\qquad\qquad\textrm{and}\qquad\qquad
\eta_{j}(z_{2,n})\rightarrow w_{2,j}
\end{equation}
as $n\rightarrow\infty$, where $j\in\{1,2,...,m\}$. Since the
index $j$ takes finite number of values we observe that one may
find a single sequence $z_{n}=(z_{1,n},z_{2,n})\in\mathbb{H}^{2}$
such that equation (17) holds for all $j\in\{1,2,...,m\}$. By
proposition 2 and lemma 3 there are $\phi_{1},\phi_{2}\in
M_{\infty}(QC(\mathbb{R}))$ such that
$\phi_{1}(\varphi_{j})=w_{1,j}$ and $\phi_{2}(\eta_{j})=w_{2,j}$
for all $j\in\{1,2,...,m\}$. Since
$\sum_{j=1}^{m}w_{1,j}w_{2,j}=w$ we have
\begin{equation*}
(\phi_{1}\hat{\otimes}\phi_{2})(\psi)=w
\end{equation*}
Therefore we have
\begin{equation*}
\mathcal{C}_{(\infty,\infty)}(\psi)\subseteq\{(\phi_{1}\hat{\otimes}\phi_{2})(\psi):\phi_{1},\phi_{2}\in
M_{\infty}(QC(\mathbb{R}))\}
\end{equation*}
For $\psi$ having an infinite sum of the following form
\begin{equation*}
\psi(z_{1},z_{2})=\sum_{j=1}^{\infty}\varphi_{j}(z_{1})\eta_{j}(z_{2})
\end{equation*}
one may choose subsequences of $\{z_{1,n}\}$ and $\{z_{2,n}\}$
through a Cantor diagonalization argument so that equation (17)
holds for all $j\in\mathbb{N}$. The rest follows in the same way
as above.
\end{proof}
We are now ready to state and prove our first main result for
quasi-parabolic composition operators acting on
$H^{2}(\mathbb{H}^{2})$:
\begin{thma}
Let $\varphi:\mathbb{H}^{2}\rightarrow$ $\mathbb{H}^{2}$ be an
analytic self-map of $\mathbb{H}^{2}$ such that
$$\varphi(z_{1},z_{2})=(z_{1}+\psi_{1}(z_{1},z_{2}),z_{2}+\psi_{2}(z_{1},z_{2})) $$
where $\psi_{j}\in$ $H^{\infty}(\mathbb{H}^{2})$ with
$\Im(\psi_{j}(z_{1},z_{2}))
> \epsilon> 0$ for all $(z_{1},z_{2})\in$ $\mathbb{H}^{2}$, $j=1,2$. Then\\
$C_{\varphi}:$ $H^{2}(\mathbb{H}^{2})\rightarrow$
$H^{2}(\mathbb{H}^{2})$ is bounded. Moreover if $\psi_{j}\in$
$(QC(\mathbb{R})\otimes QC(\mathbb{R}))\cap
H^{\infty}(\mathbb{H}^{2})$ then we have
\[\sigma_{e}(C_{\varphi})\supseteq\{e^{i(z_{1}t_{1}+z_{2}t_{2})}:t_{1},t_{2}\in [0,\infty],z_{1}\in\mathcal{C}_{(\infty,\infty)}(\psi_{1})\textrm{and}\quad z_{2}\in\mathcal{C}_{(\infty,\infty)}(\psi_{2})\}\cup\{0\},\]
where\quad $\mathcal{C}_{(\infty,\infty)}(\psi)$ is the set of
cluster points of $\psi$ at $(\infty,\infty)$.
\end{thma}
\begin{proof} The boundedness of $C_{\varphi}$ is a consequence of
proposition 6. By Proposition 8 we have the following series
expansion for $C_{\varphi}$:
\begin{equation*}
    C_{\varphi} =\sum_{n,m=0}^{\infty}T_{\tau_{1}^{n}}T_{\tau_{2}^{m}}D_{\frac{(-it_{1})^{n}e^{-\alpha t_{1}}}{n!}}D_{\frac{(-it_{2})^{m}e^{-\alpha t_{2}}}{m!}},
\end{equation*}
where $\tau_{1}(z_{1},z_{2})=$ $i\alpha-\psi_{1}(z_{1},z_{2})$ and
$\tau_{2}(z_{1},z_{2})=$ $i\alpha-\psi_{2}(z_{1},z_{2})$. So we
conclude that if $\psi_{1}, \psi_{2}\in$ $QC(\mathbb{R})\otimes
QC(\mathbb{R})\cap H^{\infty}(\mathbb{H}^{2})$ with
$\Im({\psi_{j}(z_{1},z_{2})})>$ $\epsilon > 0$, $j=1,2$, then
\begin{equation*}
C_{\varphi}\in\Psi(QC(\mathbb{R}^{2}),C_{0}((\mathbb{R}^{+})^{2}))=\Psi_{2}
\end{equation*}
where $\varphi(z_{1},z_{2})=$
$(z_{1}+\psi_{1}(z_{1},z_{2}),z_{2}+\psi_{2}(z_{1},z_{2}))$. We
look at the values $\phi(C_{\varphi})$ of $\phi$ where $\phi\in
M(\Psi_{2}/K(H^{2}(\mathbb{H}^{2})))$. By Proposition 9 we have
\begin{equation*}
M(\Psi_{2}/K(H^{2}(\mathbb{H}^{2})))=\mathbb{M}\times\mathbb{M}
\end{equation*}
where $\mathbb{M}$ is the maximal ideal space of the C*-algebra
$\Psi(QC(\mathbb{R}),C_{0}(\mathbb{R}^{+}))/K(H^{2}(\mathbb{H}))$
generated by Toeplitz operators with $QC(\mathbb{R})$ symbols and
continuous Fourier multipliers modulo compact operators acting on
$H^{2}(\mathbb{H})$. Let $$\phi=\phi_{1}\hat{\otimes}\phi_{2}\in
M(\Psi_{2}/K(H^{2}(\mathbb{H})))$$ where
$\phi_{1},\phi_{2}\in\mathbb{M}$ as in the identification done in
equation (8). If $\phi_{1}=(x_{1},\infty)$ or
$\phi_{2}=(x_{2},\infty)$ where $x_{1},x_{2}\in\
M(QC(\mathbb{R}))$ then we have
\begin{equation*}
 \phi(C_{\varphi})=\sum_{n,m=0}^{\infty}\frac{1}{n!m!}\hat{\tau_{1}}(x_{1},x_{2})^{n}\hat{\tau_{2}}(x_{1},x_{2})^{m}\vartheta_{1,n}(\infty,t_{2})\vartheta_{2,m}(\infty,t_{2})=0\\
\end{equation*}
$\forall x_{1},x_{2}\in M(QC(\mathbb{R}))$ and $t_{2}\in
[0,\infty)$ since $\vartheta_{1,n}(\infty,t_{2})=0$ for all
$n\in\mathbb{N}$ where
$\vartheta_{1,n}(t_{1},t_{2})=(-it_{1})^{n}e^{-\alpha t_{1}}$ and
$\vartheta_{2,m}(t_{1},t_{2})=(-it_{2})^{m}e^{-\alpha t_{2}}$. If
$\phi_{1}=(x_{1},t_{1})$ and $\phi_{2}=(x_{2},t_{2})$ where
$x_{1},x_{2}\in M_{\infty}(QC(\mathbb{R}))$ and $t_{1}\neq\infty$,
$t_{2}\neq\infty$, then we have
\begin{eqnarray}
 &\nonumber &(\phi_{1}\hat{\otimes}\phi_{2})(C_{\varphi})=\sum_{n,m=0}^{\infty}\frac{1}{n!m!}\hat{\tau_{1}}(x_{1},x_{2})^{n}\hat{\tau_{2}}(x_{1},x_{2})^{m}\vartheta_{1,n}(t_{1},t_{2})\vartheta_{2,m}(t_{1},t_{2})\\
 &\nonumber &=(e^{-\alpha t_{1}}\sum_{n=0}^{\infty}\frac{1}{n!}\hat{\tau_{1}}(x_{1},x_{2})^{n}(-it_{1})^{n})(e^{-\alpha t_{2}}\sum_{m=0}^{\infty}\frac{1}{m!}\hat{\tau_{2}}(x_{1},x_{2})^{m}(-it_{2})^{m})\\
 & &= e^{i(\hat{\psi_{1}}(x_{1},x_{2})t_{1}+\hat{\psi_{2}}(x_{1},x_{2})t_{2})}
\end{eqnarray}
Since $\Psi_{2}/K(H^{2}(\mathbb{H}^{2}))$ is a closed *-subalgebra
of $B(H^{2}(\mathbb{H}^{2}))/K(H^{2}(\mathbb{H}^{2}))$ we have, by
equation (4),
\begin{equation*}
\sigma_{\Psi_{2}/K(H^{2}(\mathbb{H}^{2}))}(C_{\varphi})=\sigma_{B(H^{2}(\mathbb{H}^{2}))/K(H^{2}(\mathbb{H}^{2}))}(C_{\varphi})=\sigma_{e}(C_{\varphi})
\end{equation*}
and by equation (5) we have
\begin{equation}
\sigma_{\Psi_{2}/K(H^{2}(\mathbb{H}^{2}))}(C_{\varphi})=\sigma_{e}(C_{\varphi})\supseteq\{(\phi_{1}\hat{\otimes}\phi_{2})(C_{\varphi}):\phi_{1},\phi_{2}\in\mathbb{M}\}
\end{equation}
By proposition 10 we have
\begin{eqnarray}
 &\nonumber &\{\hat{\psi_{1}}(x_{1},x_{2}):x_{1},x_{2}\in M_{\infty}(QC(\mathbb{R}))\}\supseteq\mathcal{C}_{(\infty,\infty)}(\psi_{1}),\\
 & &\{\hat{\psi_{2}}(x_{1},x_{2}):x_{1},x_{2}\in M_{\infty}(QC(\mathbb{R}))\}\supseteq\mathcal{C}_{(\infty,\infty)}(\psi_{2})
\end{eqnarray}
Therefore by equations (19), (20) and (21) we have
\begin{eqnarray*}
 & &\sigma_{e}(C_{\varphi})\supseteq\{(\phi_{1}\hat{\otimes}\phi_{2})(C_{\varphi}):\phi_{1},\phi_{2}\in\mathbb{M}\}\supseteq\\
 & &\{e^{i(z_{1}t_{1}+z_{2}t_{2})}:t_{1},t_{2}\in [0,\infty],z_{1}\in\mathcal{C}_{(\infty,\infty)}(\psi_{1})\quad\textrm{and}\quad z_{2}\in\mathcal{C}_{(\infty,\infty)}(\psi_{2})\}\cup\{0\}
\end{eqnarray*}
\end{proof}

\begin{thmb}
Let $\varphi:\mathbb{D}^{2}\rightarrow$ $\mathbb{D}^{2}$ be an
analytic self-map of $\mathbb{D}^{2}$ such that
$$\varphi(z_{1},z_{2})=\left(\frac{2iz_{1}+\psi_{1}(z_{1},z_{2})(1-z_{1})}{2i+\psi_{1}(z_{1},z_{2})(1-z_{1})},\frac{2iz_{2}+\psi_{2}(z_{1},z_{2})(1-z_{2})}{2i+\psi_{2}(z_{1},z_{2})(1-z_{2})}\right) $$
where $\psi_{j}\in$ $H^{\infty}(\mathbb{D}^{2})$ with
$\Im(\psi_{j}(z_{1},z_{2}))
> \epsilon> 0$ for all $(z_{1},z_{2})\in$ $\mathbb{D}^{2}$, $j=1,2$. Then\\
$C_{\varphi}:$ $H^{2}(\mathbb{D}^{2})\rightarrow$
$H^{2}(\mathbb{D}^{2})$ is bounded. Moreover if $\psi_{j}\in$
$(QC\otimes QC)\cap H^{\infty}(\mathbb{D}^{2})$ then we have
\[\sigma_{e}(C_{\varphi})\supseteq\{e^{i(z_{1}t_{1}+z_{2}t_{2})}:t_{1},t_{2}\in [0,\infty],z_{1}\in\mathcal{C}_{(1,1)}(\psi_{1})\textrm{and}\quad z_{2}\in\mathcal{C}_{(1,1)}(\psi_{2})\}\cup\{0\},\]
where\quad $\mathcal{C}_{(1,1)}(\psi)$ is the set of cluster
points of $\psi$ at $(1,1)\in\mathbb{T}^{2}$.
\end{thmb}
\begin{proof}
Using the isometric isomorphism
$\Phi:H^{2}(\mathbb{D}^{2})\longrightarrow$
$H^{2}(\mathbb{H}^{2})$ introduced in section 2, if
$\varphi:\mathbb{D}^{2}\rightarrow$ $\mathbb{D}^{2}$ is of the
form
\begin{equation*}
\varphi(z_{1},z_{2})=\left(\frac{2iz_{1}+\psi_{1}(z_{1},z_{2})(1-z_{1})}{2i+\psi_{1}(z_{1},z_{2})(1-z_{1})},\frac{2iz_{2}+\psi_{2}(z_{1},z_{2})(1-z_{2})}{2i+\psi_{2}(z_{1},z_{2})(1-z_{2})}\right)
\end{equation*}
where $\psi_{j}\in$ $H^{\infty}(\mathbb{D})^{2}$ satisfies
$\Im(\psi_{j}(z_{1},z_{2}))>$ $\delta
> 0$ then, by equation (1), for $\tilde{\varphi}=$
$\mathfrak{C}_{2}^{-1}\circ\varphi\circ\mathfrak{C}_{2}$ we have
$\tilde{\varphi}(z_{1},z_{2})=(z_{1}+\psi_{1}\circ\mathfrak{C}_{2}(z_{1},z_{2}),z_{2}+\psi_{2}\circ\mathfrak{C}_{2}(z_{1},z_{2}))$
and
\begin{equation}
\Phi C_{\varphi}\Phi^{-1} =
    T_{(1+\frac{\psi_{1}\circ\mathfrak{C}_{2}(z_{1},z_{2})}{z_{1}+i})(1+\frac{\psi_{2}\circ\mathfrak{C}_{2}(z_{1},z_{2})}{z_{2}+i})}C_{\tilde{\varphi}}
\end{equation}
Since both
$T_{(1+\frac{\psi_{1}\circ\mathfrak{C}_{2}(z_{1},z_{2})}{z_{1}+i})(1+\frac{\psi_{2}\circ\mathfrak{C}_{2}(z_{1},z_{2})}{z_{2}+i})}$
and $C_{\tilde{\varphi}}$ are bounded and $\Phi$ is an isometric
isomorphism, it follows that $C_{\varphi}$ is also bounded. For
$\psi_{j}\in$ $QC\otimes QC$, $j=1,2$ we have both
\begin{equation*}
C_{\tilde{\varphi}}\in\Psi_{2}\quad\textrm{and}\quad
T_{(1+\frac{\psi_{1}\circ\mathfrak{C}_{2}(z_{1},z_{2})}{z_{1}+i})(1+\frac{\psi_{2}\circ\mathfrak{C}_{2}(z_{1},z_{2})}{z_{2}+i})}\in\Psi_{2}
\end{equation*}
and hence
\begin{equation*}
 \Phi\circ
C_{\varphi}\circ\Phi^{-1}\in\Psi_{2}.
\end{equation*}
For any $\phi_{1}\hat{\otimes}\phi_{2}\in
M(\Psi_{2}/K(H^{2}(\mathbb{H})))=\mathbb{M}\times\mathbb{M}$ we
observe that
\begin{equation*}
(\phi_{1}\hat{\otimes}\phi_{2})(T_{\frac{\psi_{1}\circ\mathfrak{C}_{2}(z_{1},z_{2})}{z_{1}+i}})=(\phi_{1}\hat{\otimes}\phi_{2})(T_{\frac{\psi_{2}\circ\mathfrak{C}_{2}(z_{1},z_{2})}{z_{2}+i}})=0
\end{equation*}
Hence we have
\begin{equation}
(\phi_{1}\hat{\otimes}\phi_{2})(\Phi\circ
C_{\varphi}\circ\Phi^{-1})=(\phi_{1}\hat{\otimes}\phi_{2})(C_{\tilde{\varphi}})
\end{equation}
By equation (9) we have
\begin{equation*}
\sigma_{e}(C_{\varphi})=\sigma_{e}(\Phi\circ
C_{\varphi}\circ\Phi^{-1})
\end{equation*}
By equations (20) and (23), we have thus
\begin{equation}
\sigma_{e}(C_{\varphi})\supseteq\{(\phi_{1}\hat{\otimes}\phi_{2})(C_{\tilde{\varphi}}):\phi_{1},\phi_{2}\in\mathbb{M}\}
\end{equation}
 By equations (21) and (24) we have
\begin{equation*}
\sigma_{e}(C_{\varphi})\supseteq\{e^{i(z_{1}t_{1}+z_{2}t_{2})}:t_{1},t_{2}\in[0,\infty],z_{1}\in\mathcal{C}_{(\infty,\infty)}(\psi_{1}\circ\mathfrak{C}_{2}),z_{2}\in\mathcal{C}_{(\infty,\infty)}(\psi_{2}\circ\mathfrak{C}_{2})\}\cup\{0\}
\end{equation*}
 Since for any $\psi\in H^{\infty}(\mathbb{D}^{2})$,
\begin{equation*}
\mathcal{C}_{(\infty,\infty)}(\psi\circ\mathfrak{C}_{2})=\mathcal{C}_{(1,1)}(\psi)
\end{equation*}
we conclude that
\[\sigma_{e}(C_{\varphi})\supseteq\{e^{i(z_{1}t_{1}+z_{2}t_{2})}:t_{1},t_{2}\in [0,\infty],z_{1}\in\mathcal{C}_{(1,1)}(\psi_{1})\textrm{and}\quad z_{2}\in\mathcal{C}_{(1,1)}(\psi_{2})\}\cup\{0\}\]
\end{proof}

% ----------------------------------------------------------------
\bibliographystyle{amsplain}
%\bibliography{}

\end{document}